\documentclass[11pt,a4paper]{amsart}[2000]
\title{Permutations of strongly self-absorbing $C^{*}$-algebras}
\author{Ilan Hirshberg}
\address{Department of Mathematics, Ben Gurion University of the Negev, P.O.B. 653, Be'er \indent Sheva 84105, Israel}
\email{ilan@math.bgu.ac.il}

\author{Wilhelm Winter}
\address{Mathematisches Institut der Universit\"at M\"unster\\
Einsteinstr. 62\\ 48149 M\"unster,
\indent Germany}

\email{wwinter@math.uni-muenster.de}

\usepackage{amssymb}
\usepackage{amsthm}
\usepackage{xypic}

\theoremstyle{plain}

\newtheorem{Thm}{Theorem}[section]
\newtheorem{Cor}[Thm]{Corollary}
\newtheorem{Lemma}[Thm]{Lemma}

\theoremstyle{definition}

\newtheorem{Rmk}[Thm]{Remark}

\newcommand{\B}{\mathcal{B}}
\newcommand{\A}{\mathcal{A}}
\newcommand{\M}{\mathcal{M}}

\newcommand{\D}{\mathcal{D}}
\newcommand{\Ch}{\mathcal{C}}

\newcommand{\Zh}{\mathcal{Z}}

\newcommand{\E}{\mathcal{E}}

\newcommand{\R}{{\mathbb R}}
\newcommand{\N}{{\mathbb N}}
\newcommand{\Z}{{\mathbb Z}}
\newcommand{\C}{{\mathbb C}}

\newcommand{\be}{{\mathbf{1}}}

\newcommand{\lb}{\left <}
\newcommand{\rb}{\right >}

\renewcommand{\arrow}{\rightarrow}

\newcommand{\aut}{\mathrm{Aut}}

\newcommand{\id}{\mathrm{id}}

\newcommand{\OO}{\mathcal{O}}

\newcommand{\MAcent}{\M(\A)_{\infty}\cap \A'}

\newcommand{\En}{\E^{\otimes n}}
\newcommand{\EnS}{\left ( \En \right )^{S_n}}

\thanks{Research partially supported by the Center for Advanced Studies in Mathematics at Ben Gurion
University. The second named author was partially supported by the
DFG (SFB 478).}

\begin{document}
\begin{abstract}
Let $G$ be a finite group acting on $\{1,...,n\}$. For any
$C^*$-algebra $\A$, this defines an action of $\alpha$ of
$G$ on $\A^{\otimes n}$. We show that if $\A$ tensorially absorbs
a UHF algebra of infinite type, the Jiang-Su algebra, or is
approximately divisible, then $\A \times_{\alpha} G$ has the
corresponding property as well.
\end{abstract}
\maketitle

\section{Introduction}
Recall (see \cite{toms-winter}) that a separable, unital infinite
dimensional $C^*$-algebra $\D$ is said to be \emph{strongly
self-absorbing} if the embedding $\D \to \D \otimes \D$ given by
$d \mapsto d \otimes 1$ is approximately unitarily equivalent to
an isomorphism (regardless of which tensor product one a-priori
uses, any such $\D$ must be nuclear, and thus there is no
ambiguity in the definition). The list of known examples of such
algebras is quite short. It consists of UHF algebras of `infinite
type' (i.e., ones for which all the primes that appear in the
supernatural number do so with infinite multiplicity), the
Jiang-Su algebra $\Zh$ (\cite{jiang-su}), the Cuntz algebras
$\OO_2$ and $\OO_{\infty}$, and tensor products of $\OO_{\infty}$
by UHF algebras of infinite type.

A $C^*$-algebra $\A$ is said to be \emph{$\D$-absorbing} for a
given strongly self-absorbing $C^*$-algebra $\D$ if $\A \cong \A
\otimes \D$. This concept plays an important role in structure
theory of $C^*$-algebras, particularly in relation to the Elliott
program.

Various permanence properties of strongly self-absorbing
$C^*$-algebras were studied in \cite{toms-winter},
\cite{hirshberg-winter:rokhlin} and
\cite{hirshberg-rordam-winter}. While the property of
$\D$-absorption does remain permanent under many constructions, it
does not pass in general to crossed products; one can find
counterexamples for $\D = \OO_2$ as well as for UHF algebras.
However, it was shown in \cite{hirshberg-winter:rokhlin} that
$\D$-absorption does pass to crossed products by $\Z$, $\R$ or by
compact groups provided the group action has a Rokhlin property
(see \cite{izumi, izumi II, izumi survey}; for the non-compact
cases, one needs to assume that $\D$ is $K_1$-injective). We do not
know whether there is an example in which $\Zh$-absorption does
not pass to crossed products.

The Rokhlin property for finite groups is very restrictive, and it
seems desirable to try to look at other examples of group actions.
In this note, we consider the following kinds of actions. Suppose
$G$ is a finite group acting on $\{1,...,n\}$. Let $\A$ be a
separable $C^*$-algebra.
 We have
an induced action $\alpha:G \to \textrm{Aut}(\A^{\otimes n})$
given by $$\alpha_g(a_1 \otimes a_2 \otimes \cdots \otimes a_n) =
a_{g(1)}\otimes a_{g(2)}\otimes\cdots \otimes a_{g(n)}$$ (any
tensor product can be used here). We note that in many cases of
interest, such an action will fail to have the Rokhlin property;
for instance, if $\A = \Zh$, then $\A^{\otimes n} \cong \Zh$ is
projectionless, and thus no action on it could have the Rokhlin
property (in other cases, in which projections do exist, the
action typically will not have the Rokhlin property, but might
have the weaker \emph{tracial Rokhlin property}, see
\cite{phillips:tracial rokhlin}).

For any $C^{*}$-algebra $\A$, we denote
$$\A_{\infty} = \ell^{\infty}(\N,\A)/\Ch_0(\N,\A) \, .$$
$\A$ may be embedded  into   $\ell^{\infty}(\N,\A)$ and into
$\A_{\infty}$ in a canonical way (as constant sequences).

If $\alpha:G \arrow \aut(\A)$ is an action of a finite group $G$
on a $C^{*}$-algebra $\A$, then we have naturally induced actions
of $G$ on $\M(\A)$, and therefore also on
$\ell^{\infty}(\N,\M(\A))$, $\M(\A)_{\infty}$ and $\MAcent$,
respectively -- those actions will all be denoted by
$\bar{\alpha}$.

The main result here is the following.

\begin{Thm}
\label{main theorem}
Let $\D$ be one of the known strongly self-absorbing $C^{*}$-algebras.
Let $\alpha$ be an action of the finite group $G$
on the separable  $C^{*}$-algebra $\A$, and
suppose $G$ acts on the set $\{1,..,n\}$. Suppose furthermore there exists a unital
homomorphism $\psi:\D^{\otimes n} \to \MAcent$  which is $G$-equivariant
with respect to the permutation action on $\D^{\otimes n}$ and the
induced action $\bar{\alpha}$ on $\M(\A)_{\infty}$.
Then, $\A \times_{\alpha} G$ is $\D$-absorbing.

Replacing $\D$ by a finite direct sum of matrix algebras other than $\C$ in
the preceding hypotheses, it follows that $\A \times_{\alpha} G$ is approximately divisible.

In particular, if $\A$ is $\D$-stable, or $\A$ is approximately divisible,
then so is $\A^{\otimes n} \times G$, where $G$ acts on $\A^{\otimes n}$ by permutation.
\end{Thm}

To prove the theorem, we shall show that for $\D$ a UHF algebra of infinite type, the
Jiang-Su algebra, or an algebra of the form $M_p\oplus M_q$ for
$p,q > n$, there exists a unital embedding of $\D$ into the fixed
point subalgebra of $\D^{\otimes n}$ under the action of the
symmetric group  $S_n$.

For $\D = \OO_2$ or $\OO_{\infty}$, this is already known (see
\cite[Theorems 4.2 and Section~5]{izumi} and \cite[Lemma~10]{KishimotoKumjian} -- for $\OO_2$, one can show that the action of $S_n$
has the Rokhlin property, and for $\OO_{\infty}$, this follows
from the much more general fact that the fixed point subalgebra of
$\OO_{\infty}$ under any finite group action is a finite direct
sum of simple purely infinite $C^*$-algebras).

We wish to point out that the arguments we have use the actual structure of
the specific strongly self-absorbing $C^*$-algebras.
It remains an open problem to find a proof that
works for any  strongly self-absorbing $C^*$-algebra.

The first author would like to thank A.~Besser, Y.~Glasner and
N.~Gurevich for some helpful pointers.

\section{Preliminaries}
If $\alpha:G \to \aut(\A)$ is an action, we shall denote by
$\A^{\alpha}$, or by $\A^G$, the fixed point subalgebra (we'll use
the latter when it is clear what the action is).

We have the following characterization of $\D$-absorption (based
on ideas of Elliott and of Kirchberg), which appears as Theorem
7.2.2 in \cite{rordam book}. We note that the statement in
\cite{rordam book} refers to the relative commutant of $\A$ in an
ultrapower of $\M(\A)$; however, it is easy to see that the
characterization still holds as stated below.

\begin{Thm}
\label{D-stable characterization} Let $\D$ be  a strongly
self-absorbing and $\A$ be any separable $C^{*}$-algebra. Then,  $\A$ is
$\D$-absorbing if and only if $\D$ admits a unital homomorphism to
$\MAcent$.
\end{Thm}
Note that since any strongly self-absorbing $C^*$-algebra has to
be simple, it follows that unless $\A = 0$, the $*$-homomorphism
above must be an embedding (of course, $0$ is $\D$-absorbing for
any $\D$).

We note that in the above theorem, if instead of strongly
self-absorbing we take $\D = M_p \oplus M_q$ for some $p,q>1$,
then $\D$-absorption should be replaced by approximate
divisibility. It is furthermore easy to show that if $\MAcent$
admits  a unital homomorphism from $M_p$, then it also admits a
unital homomorphism from the UHF algebra $M_{p^{\infty}}$.

 A simple modification of Lemma 2.3 from
\cite{hirshberg-winter:rokhlin} gives us the following.

\begin{Lemma}
\label{invariant embedding implies absorption} Let $\A$, $\D$ be unital
separable $C^{*}$-algebras. Let $G$ be a finite group and
$\alpha: G \arrow \aut(\A)$ be an action. Suppose $\D$ admits a
unital homomorphism into $ (\A_{\infty} \cap A')^{\bar{\alpha}}$.
Then, $\D$ admits a unital homomorphism into $(\M(\A
\times_{\alpha}G))_{\infty} \cap (\A \times_{\alpha}G)'$.

In particular, it follows that if $\D = M_p$ above, then $\A
\times_{\alpha} G$ absorbs $M_{p^{\infty}}$. If $\D = M_p \oplus
M_q$ for some $p,q>1$, then $\A \times_{\alpha} G$ is
approximately divisible.
\end{Lemma}

\begin{Rmk}
\label{r-equivariant-embedding}
In the situation of Theorem \ref{main theorem}, note that,
since $\psi$ is $G$-equivariant,
the $S_n$-invariant subalgebra of the copy of $\D^{\otimes n}$ in
$\MAcent$ is clearly contained in the fixed point subalgebra of
$\MAcent$ under the action of $G$. It will thus suffice for us to
show that there is a unital embedding of $\D$ into $\left (
\D^{\otimes n} \right )^{S_n}$

Henceforth, therefore, we assume that $G=S_n$, and that $\alpha$
is given by permutation of the tensor factors.
\end{Rmk}

We shall also require a result from \cite{hirshberg-rordam-winter}
concerning $C(X)$-algebras. Recall that, for a compact Hausdorff
space $X$, a $C(X)$-algebra is a $C^*$-algebra $\A$, along with a
fixed unital homomorphism from $C(X)$ to the center of $\M(\A)$.
The \emph{fiber} over $x \in X$, denoted $A_x$, is the quotient
$\A/C_0(X \setminus \{x\})\A$.

\begin{Thm}[\cite{hirshberg-rordam-winter}, Theorem 4.6]
\label{HRW thm}
Let $\D$ be a $K_{1}$-injective strongly
self-absorbing $C^{*}$-algebra. Let $X$ be a compact metrizable
space of finite covering dimension. Let $\A$ be a separable
$C(X)$-algebra. It follows that $\A$ is $\D$-absorbing if and only
if $\A_{x}$ is $\D$-absorbing for each $x \in X$.
\end{Thm}

\section{Proof of the main theorem}

\begin{Lemma}
\label{matrix fixed point}
 For some $m,n \in \N$, consider the action $\alpha$ of $S_n$ on $M_m^{\otimes n}$
given by permutation of the factors. Suppose there are $p,k \in \N$ are such that
$p$ is prime, $p^{k} | m$ and $p^{k} \not | n!$. Then, there exists a
unital embedding of $M_p$ into $\left ( M_m^{\otimes n} \right
)^{S_n}$.
\end{Lemma}
\begin{proof}
Denote $V = \C^m$. Note that $M_m^{\otimes n} \cong \B(V^{\otimes
n})$. We have a unitary representation $U$ of $S_n$ on $V^{\otimes
n}$ given by
$$U_g (\xi_1 \otimes \xi_2 \otimes \cdots \otimes
\xi_n) = \xi_{g(1)}\otimes \xi_{g(2)}\otimes\cdots \otimes
\xi_{g(n)}
$$
Note that $\alpha_g(a) = U_g a U_g^*$. Thus, the fixed point
subalgebra is
$$\{U_g \mid g \in S_n\}'\cong \bigoplus_{\rho \in
\hat{S_n}}M_{\mu(\rho)}
$$
where $\hat{S_n}$ is the set of (equivalence classes of) irreducible representations of
$S_n$, and $\mu(\rho)$ is the multiplicity of the representation
$\rho$ as a subrepresentation of $U$. It suffices, therefore, to
show that $p | \mu(\rho)$ for all $\rho \in \hat{S_n}$. To this
effect, we shall compute the character $\chi$ of the
representation $U$.

Suppose $g \in S_n$ is represented as $\ell$ disjoint cycles
(including cycles of length 1). Let $\xi_1,...,\xi_m$ be an
orthonormal basis for $V$. So, $\{\xi_{i_1} \otimes \cdots \otimes
\xi_{i_n} \mid i_1,...,i_n \in \{1,...,m\}\}$ form an orthonormal
basis for $V^{\otimes n}$. So, we have
$$\chi(g) = \sum_{i_1,...,i_n \in
\{1,...,m\}} \lb \xi_{i_1} \otimes \cdots \otimes \xi_{i_n},
\xi_{g(i_1)} \otimes \cdots \otimes \xi_{g(i_n)}\rb
$$
$$= \sum_{i_1,...,i_n \in \{1,...,m\}} \delta_{(i_1,...,i_n) ,
(g(i_1),...,g(i_n))},$$
where the last expression is simply the number of $n$-tuples of
 elements from $\{1,...,m\}$ which are fixed under the permutation
 $g$. An $n$-tuple is fixed if and only if it is constant on each
 of the cycles of $g$, and there are $\ell$ of those. Thus
  $$\chi(g) = m^{\ell}.$$
Recall that the characters of the irreducible representations of
$S_n$ are all integer-valued (this follows immediately from the
Frobenius character formula; see, for instance, chapter 4 of
\cite{fulton-harris}). Thus, suppose $\rho$ is an irreducible
representation, and its character is $\chi_{\rho}$, then
$$\mu(\rho) = \frac{1}{n!} \lb \chi_{\rho},\chi \rb$$
Since all the entries of $\chi$ are integers divisible by $m$, and
hence by $p^k$, and $\chi_{\rho}$ consists of integer entries, we
see that $\lb \chi_{\rho},\chi \rb$ is an integer and $p^k|\lb
\chi_{\rho},\chi \rb$, and since $p^k \not | n!$, we have that $p
| \mu(\rho)$, as required.
\end{proof}

The following is now immediate.

\begin{Cor}
\label{UHF absorption}  If $\D$ is a strongly self absorbing UHF
algebra, and $p$ is a prime which appears in the supernatural
number associated to $\D$, then there is a unital embedding of
$M_{p^{\infty}}$ into the fixed point subalgebra of $\D^{\otimes n}$ under
the action of $S_{n}$.
\end{Cor}

\begin{proof}
The preceding lemma clearly yields an embedding of $M_{p}$ into the fixed point algebra
of $\D^{\otimes n}$, hence an embedding of $M_{p^{\infty}}$ into the fixed point algebra
of $(\D^{\otimes \infty})^{n}$. Identifying $\D$ with $\D^{\otimes \infty}$,
we obtain the result.
\end{proof}

In view of Remark \ref{r-equivariant-embedding} and
Lemma \ref{invariant embedding implies absorption},
we thus proved Theorem \ref{main theorem} for the case of UHF
algebras of infinite type.

For $\D$ the tensor product of $\OO_{\infty}$ with a UHF algebra $\B$ of infinite type,
we may apply Theorem \ref{main theorem} for $\OO_{\infty}$ and $\B$ separately,
using $\psi \circ (\id_{\OO_{\infty}} \otimes \be_{\B})^{\otimes n}$ and
$\psi \circ (\be_{\OO_{\infty}} \otimes \id_{\B})^{\otimes n}$, respectively,
in place of $\psi$ to show that $\A \times_{\alpha} G$ absorbs both $\OO_{\infty}$ and $\B$.
But then it is straightforward to conclude from Theorem \ref{D-stable characterization} that
$\A \times_{\alpha} G$ also absorbs $\OO_{\infty} \otimes \B$.

We now turn to the case of approximate divisibility. We first need
a simple technical lemma. The proof is done via a standard
diagonalization trick, which we leave to the reader. Any tensor
product can be taken in the lemma below.

\begin{Lemma}
\label{cantor-diag} Let $\A$, $\D$ be  $C^*$-algebras, and let
$\alpha:G \to \aut(\A)$ be an action of a finite group. Suppose
$\D$ is separable, and let $\mathcal{C}$ be a separable subspace
of $\MAcent$. Suppose there exists a unital homomorphism
$$\psi:\D^{\otimes n} \to \MAcent$$ such that
$$\bar{\alpha}_g(\psi(d_1 \otimes d_2 \otimes \cdots  \otimes d_n)) =
\psi(d_{g(1)}\otimes d_{g(2)} \otimes \cdots \otimes d_{g(n)})$$ for all
$g \in G$ and $d_1,...,d_n$ in $\D$. Then, there exists a unital
homomorphism $$\psi':\D^{\otimes n} \to \MAcent\cap \mathcal{C}'$$
which satisfies the same property with respect to the action
$\bar{\alpha}$.
\end{Lemma}

\begin{Cor}
\label{approximate divisibility permanence}
Under the hypotheses
of Theorem \ref{main theorem}, for $\D$ a finite direct sum of
matrix algebras other than $\C$, we have that $\A \times_{\alpha}
G$ is approximately divisible.
\end{Cor}
\begin{proof}
Choose two different prime numbers $p,q$ greater than $n$. By a
repeated application of Lemma \ref{cantor-diag}, we see that for
any $m$, we can find unital homomorphisms
$\psi_1,...,\psi_m:\D^{\otimes n} \to \MAcent$ with commuting
ranges such that
$$\bar{\alpha}_g(\psi_k(d_1 \otimes d_2 \otimes \cdots \otimes d_n)) =
\psi_k(d_{g(1)}\otimes d_{g(2)} \otimes \cdots \otimes d_{g(n)})$$ for all
$k=1,...,m$, $g \in G$ and $d_1,...,d_n$ in $\D$. Define
$\psi:\left ( \D^{\otimes m} \right )^{\otimes n} \to \MAcent$ by
$\psi = \psi_1 \otimes \psi_2 \otimes \cdots \otimes \psi_m$, then
we have
$$\bar{\alpha}_g(\psi(d_1 \otimes d_2 \otimes \cdots d_n)) =
\psi(d_{g(1)}\otimes d_{g(2)} \otimes \cdots d_{g(n)})$$ for all
$g \in G$ and $d_1,...,d_n$ in $\D^{\otimes m}$. Of course, if
$\D_0$ is any unital subalgebra of $\D^{\otimes m}$, we could
restrict $\psi$ to $\D_0$ to get an embedding with similar
properties, using $\D_0$ instead of $\D$.

We recall that there is a natural number $N$ such that for any
natural number $k \geq N$ there are positive integers $a,b$ such
that $ap+bq = k$, and therefore there is a unital embedding of
$M_p \oplus M_q$ into $M_k$. Now, if $\ell>1$ is the size of the
smallest matrix algebra summand in $\D$, the size of the smallest
matrix algebra summand in $\D^{\otimes m}$ is $\ell^m$, and thus,
for sufficiently large $m$, there exists a unital embedding of
$M_p\oplus M_q$ into $\D^{\otimes m}$.  We may therefore assume
that we had $\D \cong M_p \oplus M_q$ to begin with.

We now show that there is a unital embedding of $M_p \oplus M_q$
into $\left ( (M_p \oplus M_q)^{\otimes n} \right )^{S_n}$. Note
that we can obviously identify
$$
(M_p \oplus M_q)^{\otimes n} \cong \bigoplus_{v \in \{0,1\}^n}
\bigotimes_{i=1}^{n}M_{p^{v(i)}} \otimes M_{q^{1-v(i)}}
$$
(where we write elements of $\{0,1\}^n$ as functions from
$\{1,2,...,n\}$ to $\{0,1\}$).
Consider the algebra
$$
\B = \bigoplus_{k=0}^n \left ( M_{p}^{\otimes n-k} \right
)^{S_{n-k}} \otimes \left ( M_{q}^{\otimes k} \right )^{S_k}
$$
where the superscript denotes that we are considering the fixed
point subalgebra under the corresponding action of the symmetric
subgroup of $S_n$. For $b \in \B$, we denote by $b(k)$ the
component of $b$ in the summand
$$\left ( M_{p}^{\otimes n-k} \right
)^{S_{n-k}} \otimes \left ( M_{q}^{\otimes k} \right )^{S_k}
\subseteq  M_{p}^{\otimes n-k}  \otimes M_{q}^{\otimes k}.$$

We have a unital embedding of $\B$ into $\bigoplus_{v \in
\{0,1\}^n} \bigotimes_{i=1}^{n}M_{p^{1-v(i)}} \otimes M_{q^{v(i)}}$ defined as follows. For
$$
a \in \bigoplus_{v \in
\{0,1\}^n} \bigotimes_{i=1}^{n}M_{p^{1-v(i)}} \otimes M_{q^{v(i)}}
$$
we denote by $a(v)$ the $v$'th component of $a\in
(M_p \oplus M_q)^{\otimes n}$, $v \in \{0,1\}^n$. Write $k_v
 = \sum_{i=1}^nv(i)$, and let  $w_v \in \{0,1\}^n$ be given by
$$w_v(i) = \left \{ \begin{matrix} 0 \mid i \leq n-k \\ 1 \mid i>n-k
\end{matrix} \right . $$

Let $g \in S_n$ be a permutation such that $g(w_v) = v$, where we
consider the naturally induced permutation action of $S_n$ on
$\{0,1\}^{\otimes n}$. We define
$$
\beta_g:  M_{p}^{\otimes n-k_v}  \otimes M_{q}^{\otimes k_v}
\to
\bigotimes_{i=1}^{n}M_{p^{1-v(i)}} \otimes M_{q^{v(i)}}
$$
by
$$
\beta_g(a_1 \otimes a_2 \otimes \cdots \otimes a_n) =
a_{g(1)} \otimes \cdots \otimes a_{g(n)}
$$
We note that if
$x \in \left ( M_{p}^{\otimes n-k} \right
)^{S_{n-k}} \otimes \left ( M_{q}^{\otimes k} \right )^{S_k}
\subseteq  M_{p}^{\otimes n-k_v}  \otimes M_{q}^{\otimes k_v}$
and $g(w_v) = h(w_v) = v$ then $\beta_g(x) = \beta_h(x)$.

We can define
$$\psi:\B \to
\bigoplus_{v \in \{0,1\}^n}
\bigotimes_{i=1}^{n}M_{p^{1-v(i)}}
\otimes M_{q^{v(i)}}$$ by
$$
\psi(b)(v) = \beta_{g_v}(b(k_v))
$$
where $g_v$ is an element of $S_n$ which satisfies that
$g_v(w_v) = v$. In particular, if $v = w_v$, then we simply have
$\psi(b)(v) = b(k_v)$.

For any $g \in S_n$, we have
$$\alpha_g(\psi(b))(v) =
\beta_{h_1}\beta_{h_2}^{-1}(\psi(b)(g^{-1}(v)))$$
where $h_1,h_2$ are such that $h_2(g^{-1}(v)) = w_{g^{-1}(v)} = w_v$,
and $h_1(w_v) = v$. Note further that
$\psi(b)(g^{-1}(v)) = \beta_{h_2}b(k_{g^{-1}(v)})$ and that
$k_{g^{-1}(v)} = k_v$, and thus
$$\alpha_g(\psi(b))(v) = \beta_{h_1}\beta_{h_2}^{-1}(\beta_{h_2}b(k_v)) =
\beta_{h_1}(b(k_v)) = \psi(b)(v),$$
so indeed the image of $\psi$
is fixed under $\alpha_g$   for all $g \in G$.

It thus suffices to construct a unital embedding
$\varphi:M_p \oplus M_q \to \B$.
For this, it will suffice, for each $k=0,1,...,n$, to construct a
unital homomorphism
$$\varphi_k : M_p \oplus M_q \to \left (
M_{p}^{\otimes n-k} \right )^{S_{n-k}} \otimes \left (
M_{q}^{\otimes k} \right )^{S_k}.$$ By Lemma \ref{matrix
fixed point}, for $k<n$, we have a unital embedding of $M_p$ into
$\left (
M_{p}^{\otimes n-k} \right )^{S_{n-k}} \otimes \C 1 \subseteq
\left (
M_{p}^{\otimes n-k} \right )^{S_{n-k}} \otimes \left (
M_{q}^{\otimes k} \right )^{S_k}$, and thus we may select $\varphi_k$ to be
such a non-injective unital homomorphism, which annihilates the summand $M_q$.
For $n=k$, we have that   $\left (
M_{p}^{\otimes n-k} \right )^{S_{n-k}} \otimes \left (
M_{q}^{\otimes k} \right )^{S_k} \cong \left (
M_{q}^{\otimes n} \right )^{S_n}$, and again by Lemma  \ref{matrix
fixed point} we can choose a unital embedding of $M_q$ into this summand,
and by annihilating the summand $M_p$, we have a unital homomorphism from
$M_p \oplus M_q$. Put together, we get a unital homomorphism
$\varphi = \bigoplus_{k=0}^n \varphi_k$, which has a trivial kernel.

\end{proof}

We now turn to the proof for $\D = \Zh$.

\begin{proof}[Proof of Theorem \ref{main theorem} for the case $\D
= \Zh$]

We denote
$$
\E = \{f \in C([0,1],M_{2^{\infty}} \otimes M_{3^{\infty}}) \mid
f(0) \in M_{2^{\infty}} \otimes 1 \;,\; f(1) \in 1 \otimes
M_{3^{\infty}} \}
$$
By Proposition 2.2 from \cite{rordam stable rank}, we know that
one can embed $\E$ unitally into $\Zh$. Thus, it suffices for us
to construct a unital homomorphism from $\Zh$ to $\left (
\E^{\otimes n} \right )^{S_n}$.

$\E$ can naturally be regarded as a $C([0,1])$-algebra (the center
of $\E$ can be identified in the obvious way with $C([0,1])$),
where the fiber $\E_0$  is $M_{2^{\infty}}$, the fiber $\E_1$ is
$M_{3^{\infty}}$, and the fibers $\E_t$ for $0<t<1$ are
$M_{2^{\infty}} \otimes M_{3^{\infty}}$. We may thus regard
$\E^{\otimes n}$ as a $C([0,1]^n)$-algebra, where the fiber over
$\vec{t} = (t_1,t_2,...,t_n)$ is isomorphic to $\E_{t_1} \otimes
\E_{t_2} \otimes \cdots \otimes \E_{t_n}$ (see Proposition 1.6 of
\cite{hirshberg-rordam-winter} for a discussion of these matters).
We shall denote $\E_{\vec{t}} = \E_{t_1} \otimes \E_{t_2} \otimes
\cdots \otimes \E_{t_n}$.

The unital inclusion $$\En \supseteq \EnS \supseteq \left (
C([0,1])^{\otimes n} \right )^{S_n} \cong C([0,1]^n/S_n) \cong
C(\Delta)$$ gives $\EnS$ and $\En$ the structure of
$C(\Delta)$-algebras, where
$$\Delta = \{(t_1,t_2,...,t_n) \in [0,1]^n \mid
t_1\leq t_2 \leq \cdots \leq t_n\}.$$

Let $\vec{t} \in \Delta$ be a point, and let $H$ be the isotropy
group of $\vec{t}$. Note that a function $f \in
C_0(\Delta\smallsetminus\{\vec{t}\})$, thought of as a function on
$[0,1]^n$, is a function which vanishes on the $S_n$-orbit of
$\vec{t}$ (i.e. on $|G/H|$ points), and hence,
 thought of as a $C(\Delta)$-algebra, we
have $\En_{\vec{t}} \cong \bigoplus_{\vec{s} \in S_n(\vec{t})}
\E_{\vec{s}}$. As the $S_n$ action drops to this quotient, we have
that $\EnS_{\vec{t}} \subseteq \En_{\vec{t}} \cong
\bigoplus_{\vec{s} \in S_n(\vec{t})} \E_{\vec{s}}$ and consists of
the $S_n$-invariants elements there. Each $S_n$-invariant element
in $\bigoplus_{\vec{s} \in S_n(\vec{t})} \E_{\vec{s}}$ is
determined by its summand in $\E_{\vec{t}}$, and thus
$\EnS_{\vec{t}}$ is isomorphic to $\E_{\vec{t}}^H$.

Denoting $\{s_1, \cdots,s_{\ell}\} = \{t_1,\cdots,t_n\}$, where
$s_1 < s_2 < \cdots < s_{\ell}$ and each $s_j$ appears $k_j$ times
in the ordered $n$-tuple $(t_1,\cdots,t_n)$, we have that $H$ is a
direct product of the symmetric groups $S_{k_j}$, $j=1,...,\ell$,
and $$\E_{\vec{t}}^H \cong \bigotimes_{j=1}^{\ell} \left (
\E_{s_j}^{\otimes k_j} \right )^{S_{k_j}}.$$

It  follows from Corollary \ref{UHF absorption} that $\left (
\E_{s_j}^{\otimes k_j} \right )^{S_{k_j}}$ absorbs a UHF algebra
(of type $2^{\infty}$ or $3^{\infty}$), and therefore, so does
$\bigotimes_{j=1}^{\ell} \left ( \E_{s_j}^{\otimes k_j} \right
)^{S_{k_j}}$. By \cite[Theorem~5]{jiang-su}, any infinite
dimensional UHF algebra is $\Zh$-absorbing, and therefore,
$\E_{\vec{t}}^H$ is $\Zh$-absorbing as well.


We thus see that all the fibers of the $C(\Delta)$-algebra $\EnS$
are $\Zh$-absorbing. From Theorem \ref{HRW thm}, since $\Delta$
has finite covering dimension, we see that $\EnS$ must be
$\Zh$-absorbing, and in particular, admits a unital embedding of
$\Zh$. Therefore, there is a unital embedding of $\Zh$ into $\left
( \Zh^{\otimes n} \right )^{S_n}$, as required.
\end{proof}

\end{document}